\documentclass[aop,MSNbibl,seceqn,dvips]{arximspdf}
\usepackage{mathbh}

\doi{10.1214/11-AOP708}
\volume{41}
\issue{3A}
\pubyear{2013}
\firstpage{1191}
\lastpage{1217}

\makeatletter
\newtheorem{theorem}{Theorem}
\newtheorem{corollary}{Corollary}
\newtheorem{proposition}{Proposition}
\newtheorem{lemma}{Lemma}
\newproclaim{example}{Example}
\newcommand{\p}{\mathbb{P}}
\newcommand{\e}{\mathbb{E}}
\newcommand{\ind}{\mathbh{1}}
\newcommand{\ed}{\stackrel{(d)}{=}}
\newcommand{\eqdef}{\stackrel{\mathrm{def}}{=}}
\makeatother

\begin{document}
\begin{frontmatter}

\title{On the law of the supremum of L\'evy processes\thanksref{T2}}
\runtitle{Supremum of L\'evy processes}
\thankstext{T2}{Supported by the ECOS-CONACYT CNRS Research Project M07-M01.}

\begin{aug}
\author[A]{\fnms{L.} \snm{Chaumont}\corref{}\thanksref{t1}\ead[label=e1]{loic.chaumont@univ-angers.fr}}
\thankstext{t1}{Ce travail a b\'en\'efici\'e d'une aide de l'Agence
Nationale de la Recherche
portant la r\'ef\'erence ANR-09-BLAN-0084-01.}
\runauthor{L. Chaumont}
\affiliation{Universit\'e d'Angers}
\address[A]{LAREMA---UMR CNRS 6093\\
Facult\'{e} des Sciences\\
Universit\'e d'Angers\\
2, Bd Lavoisier 49045\\
Angers Cedex 01\\
France} 
\end{aug}

\received{\smonth{11} \syear{2010}}
\revised{\smonth{8} \syear{2011}}

%
\begin{abstract}
We show that the law of the overall supremum $\overline{X}_t=\sup
_{s\le
t}X_s$ of a L\'evy process $X$, before the
deterministic time $t$ is equivalent to the average occupation measure
$\mu_t^+(dx)=\int_0^t\p(X_s\in dx) \,ds$, whenever
0 is regular for both open halflines $(-\infty,0)$ and $(0,\infty)$. In
this case, $\p(\overline{X}_t\in dx)$ is absolutely continuous
for some (and hence for all) $t>0$ if and only if the resolvent measure
of $X$ is absolutely continuous. We also study the cases
where 0 is not regular for both halflines. Then we give absolute
continuity criterions for
the laws of $(g_t,\overline{X}_t)$ and $(g_t,\overline{X}_t,X_t)$,
where $g_t$ is the time at which
the supremum occurs before $t$. The proofs of these results use an
expression of the~joint law
$\p(g_t\in ds,X_t\in dx,\overline{X}_t\in dy)$ in terms of the entrance
law of
the excursion measure of the reflected process at the supremum and that
of the reflected process at the infimum.
As an application, this law is made (partly) explicit in some
particular instances.
\end{abstract}

%

\end{frontmatter}
%

\section{Introduction}\label{sec1}

The law of the past supremum $\overline{X}_t=\sup_{s\le t}X_s$ of L\'
evy processes before a deterministic time $t>0$
presents some major interest in stochastic modeling, such as queuing
and risk theories, as it is related to the law of
the first passage time $T_x$ above any positive level $x$, through the
relation $\p(\overline{X}_t\ge x)=\p(T_x\le t)$. The
importance of knowing features of this law, for some domains of
application, mainly explains the abundance of the literature on
this topic. From the works of L\'evy
on Brownian motion~\cite{le} to the recent developments of Kuznetsov~\cite{ku2} for a very large class of
stable L\'evy processes, an important number of papers have appeared.
Most of them concern explicit computations for
stable processes and basic features, such as tail behavior of this law,
are still unknown in the general case.

The present work is mainly concerned with the study of the nature of
the law of the overall supremum $\overline{X}_t$
and, more specifically, with the existence of a density for this
distribution. In a recent paper,
Bouleau and Denis~\cite{bd} proved that the law of $\overline{X}_t$ is
absolutely continuous
whenever the L\'evy measure of $X$ is itself absolutely continuous and
satisfies some additional conditions; see Proposition 3
in~\cite{bd}. This result has raised our interest on the subject, and
we propose to determine ``exploitable'' necessary
and sufficient conditions, under which the law of $\overline{X}_t$ is
absolutely continuous. Doing so, we also obtained conditions
for the absolute continuity of the random vectors $(g_t,\overline
{X}_t)$ and $(g_t,\overline{X}_t,X_t)$,
where $g_t$ is the time at which the maximum of $X$ occurs on $[0,t]$.
The proofs are based on two main ingredients. The first
one is the equivalence between the law of $X_t$ in $\mathbb{R}_+$ and
the entrance law of the excursions of the reflected process
at its minimum; see Lemma~\ref{equivalence1}. The second argument is an
expression of the law of $(g_t,\overline{X}_t,X_t)$ in
terms of the entrance laws $q_t$ and $q_t^*$ of the excursions of both
reflected processes, at the maximum and at the minimum,
respectively: if 0 is regular for both half lines $(-\infty,0)$ and
$(0,\infty)$, then
\[
\p(g_t\in ds,\overline{X}_t\in dx,\overline{X}_t-X_t\in
dy)=q_s^*(dx)q_{t-s}(dy)\ind_{[0,t]}(s) \,ds.
\]
This expression is extended to the nonregular cases in Theorem \ref
{law}. As another application, we may recover
the law of the triplet $(g_t,\overline{X}_t,X_t)$ for Brownian motion
with drift and derive an explicit form of this law, for the symmetric
Cauchy process. The law of $(g_t,\overline{X}_t)$, may also be computed
in some instances of spectrally negative L\'evy processes.

The remainder of this paper is organized as follows. In Section \ref
{prelim}, we give some definitions and we recall some basic
elements of excursion theory and fluctuation theory for L\'evy
processes, which are necessary for the proofs. The main results
of the paper are stated in Sections~\ref{main} and~\ref{expressions}.
In Section~\ref{main}, we state continuity properties
of the triple $(g_t,\overline{X}_t,X_t)$, whereas Section \ref
{expressions} is devoted
to some representations and explicit expressions for the law of this triple.
Then the proofs of the results are postponed to Section~\ref{proofs}.

\section{Preliminaries}\label{prelim}

We denote by $\mathcal{D}$ the space of c\`{a}dl\`{a}g paths $\omega:[0,\infty)
\rightarrow\mathbb{R\cup\{\infty\}}$ with lifetime $\zeta
(\omega)=\inf\{t\ge0:\omega_{t}=\omega_s, \forall s\ge t\}$, with
the usual convention that $\inf\{\varnothing\}=+\infty$.
The space $\mathcal{D}$ is equipped with the Skorokhod topology, its
Borel $\sigma$-algebra $\mathcal{F}$ and the usual completed
filtration \mbox{$(\mathcal{F}_{s},s\geq0)$}, generated by the
coordinate process $X=(X_{t},t\geq0)$ on
the space~$\mathcal{D}$. We write $\overline{X}$ and $\underline{X}$
for the
supremum and infimum processes,
\[
\overline{X}_{t}=\sup\{X_{s}:0\leq s\leq t\}\quad \mbox{and}\quad
\underline{X}_{\hspace*{1pt}t}=\inf\{X_{s}:0\leq s\leq t\}.
\]
For $t>0$, the last passage times by $X$ at its supremum and at its
infimum before $t$ are, respectively, defined by
\begin{eqnarray*}
g_t&=&\sup\{s\le t:X_s=\overline{X}_t \mbox{ or } X_{s-}=\overline
{X}_t\}
\quad\mbox{and}\\
 g_t^*&=&\sup\{s\le t:X_s=\underline{X}_{\hspace*{1pt}t} \mbox{ or }
X_{s-}=\underline{X}_{\hspace*{1pt}t}\}.
\end{eqnarray*}
We also define the first passage time by $X$ in the open halfline
$(0,\infty)$ by
\[
\tau_0^+=\inf\{t\ge0:X_t>0\}.
\]
We denote by $\mathbb{P}$ the law on $\mathcal{D}$ of a L\'{e}vy
process starting from $0$. When $(X,\p)$ or $(-X,\p)$ is a subordinator,
the past supremum at time $t$ corresponds to the value $X_t$ or 0,
respectively. So these cases will be excluded in the sequel.
Besides, the technics which are used in this paper are not quite
adapted to the case of compound Poisson processes which will
be treated apart, in Theorem~\ref{new}. So unless explicitly mentioned,
in the sequel, we assume that $X$ is not a compound
Poisson process.

Note that under our assumptions, 0 is always regular for $(-\infty,0)$
or/and $(0,\infty)$.
It is well known that the reflected processes $\overline{X}-X$ and
$X-\underline{X}$ are strong
Markov processes. Under $\mathbb{P}$, the state 0 is regular for
$(0,\infty)$ [resp., for $(-\infty,0)$]
if and only if it is regular for $\{0\}$, for the reflected process
$\overline{X}-X$ (resp., for $X-\underline{X}$).
If 0 is regular for $(0,\infty)$, then the local time at 0 of the
reflected process $\overline{X}-X$ is the unique
continuous, increasing, additive functional $L$ with $L_0=0$, a.s.,
such that the support of the measure $dL_t$ is
the set $\overline{\{t:\overline{X}_t=X_t\}}$ and which is normalized by
%
%
\begin{equation}\label{norm1}
\e\biggl(\int_0^\infty e^{-t} \,dL_t\biggr)=1.
\end{equation}
Let $G$ be the set of the left endpoints of the excursions away from 0
of $\overline{X}-X$, and for each $s\in G$,
call $\varepsilon^s$ the excursion which starts at $s$. Denote by $E$ the
set of excursions, that is,
$E=\{\omega\in\mathcal{ D}:\omega_t>0, \mbox{for all $0<t<\zeta(\omega
)$}\}
$, and let $\mathcal{ E}$ be the Borel $\sigma$-algebra which
is the trace of $\mathcal{ F}$ on the subset $E$ of~$\mathcal{ D}$.
The It\^o measure $n$ of the excursions away from 0 of the process
$\overline{X}-X$ is characterized by the
so-called \textit{compensation formula},
%
%
\begin{equation}\label{compensation}\e\biggl(\sum_{s\in G}F(s,\omega,\varepsilon^s)\biggr)=
\e\biggl(\int_0^\infty \,dL_s\biggl(\int F(s,\omega,\varepsilon)n(d\varepsilon
)\biggr)\biggr),
\end{equation}
which is valid whenever $F$ is a positive and predictable process, that
is, $\mathcal{ P}(\mathcal{ F}_s)\otimes\mathcal{ E}$-measurable,
where $\mathcal{ P}(\mathcal{ F}_s)$ is the predictable $\sigma$-algebra
associated to the filtration $(\mathcal{ F}_s)$.
We refer to~\cite{be}, Chapter~IV,~\cite{ky}, Chapter~6 and~\cite{do}
for more detailed definitions and some
constructions of $L$ and $n$.

If 0 is not regular for $(0,\infty)$, then the set $\{t:(\overline
{X}-X)_t=0\}$ is discrete, and following
\cite{be} and~\cite{ky}, we define the local time $L$ of $\overline
{X}-X$ at 0 by
%
%
\begin{equation}\label{norm2}
L_t=\sum_{k=0}^{R_t}\mathbf{ e}^{(k)},
\end{equation}
where $R_t=\operatorname{Card}\{s\in(0,t]:\overline{X}_s=X_s\}$, and
$\mathbf{e}^{(k)}$, $k=0,1,\ldots$ is a sequence of independent and exponentially
distributed random variables with parameter
%
%
\begin{equation}\label{alpha}
\gamma=\bigl(1-\e(e^{-\tau^+_0})\bigr)^{-1}.
\end{equation}
In this case, the measure $n$ of the excursions away from 0 is
proportional to
the distribution of the process $X$ under the law $\mathbb{P}$,
returned at its first passage
time in the positive halfline. More formally, let us define
$\varepsilon^{\tau_0^+}=(-X_s,0\le s<\tau_0^+)$,
then for any bounded Borel functional $K$ on $\mathcal{ E}$,
%
%
\begin{equation}\label{excdisc}
\int_\mathcal{ E}K(\varepsilon)n(d\varepsilon)=\gamma\e[K(\varepsilon^{\tau
_0^+})].
\end{equation}
Define $G$ and $\varepsilon^s$ as in the regular case. Then from
definitions (\ref{norm2}), (\ref{excdisc}) and an application
of the strong Markov property, we may check that the normalization~(\ref
{norm1}) and the compensation formula (\ref{compensation})
are still valid in this case.

The local time at 0 of the reflected process at its infimum
$X-\underline{X}$, and the measure of its excursions away from 0
are defined in the same way as for $\overline{X}-X$. They are
respectively denoted by $L^*$ and $n^*$.
Then the ladder time processes $\tau$ and $\tau^*$, and the ladder
height processes $H$ and $H^*$ are the following
(possibly killed) subordinators:
\begin{eqnarray*}
\tau_t&=&\inf\{s:L_s>t\},\qquad \tau^*_t=\inf\{s:L_s^*>t\},\\
H_t&=&X_{\tau_t}, \qquad H^*_t=-X_{\tau_t^*},\qquad t\ge0,
\end{eqnarray*}
where $\tau_t=H_t=+\infty$, for $t\ge\zeta(\tau)=\zeta(H)$ and
$\tau
_t^*=H_t^*=+\infty$, for $t\ge\zeta(\tau^*)=\zeta(H^*)$.
The characteristic exponent $\kappa$ of the ladder process $(\tau,H)$
may be defined by
\[
\e\biggl(\int_0^\infty \,dL_te^{-qt}\exp(-\alpha t-\beta\overline
{X}_t)\biggr)=\frac1{\kappa(q+\alpha,\beta)},\qquad q>0, \alpha,\beta\ge0.
\]
From (\ref{norm1}), we derive that $\kappa(1,0)=\kappa^*(1,0)=1$, so
that the Wiener--Hopf factorization in time (which is stated
in~\cite{be}, page 166 and in~\cite{ky}, page 166) is normalized as follows:
%
%
\begin{equation}\label{wh}
\kappa(\alpha,0)\kappa^*(\alpha,0)=\alpha,\qquad \mbox{for all $\alpha
\ge0$.}
\end{equation}
Recall also that the drifts $\mathtt{d}$ and $\mathtt{d}^*$ of the
subordinators $\tau$ and $\tau^*$ satisfy $\mathtt{d}=0$
(resp., $\mathtt{d}^*=0$) if and only if 0 is regular for $(-\infty, 0)$
[resp., for $(0,\infty)$], and that
%
%
\begin{equation}\label{delta}
\int_0^t\ind_{\{X_s=\overline{X}_s\}} \,ds=\mathtt{d} L_t \quad\mbox{and}\quad
\int_0^t\ind_{\{X_s=\underline{X}_s\}} \,ds=\mathtt{d}^*L_t^*.
\end{equation}
Suppose that 0 is not regular for $(0,\infty)$, and let $\mathbf{ e}$ be
an independent exponential time with mean 1,
then from (\ref{norm1}) and (\ref{delta}), $\p((X-\underline
{X})_\mathbf{e}=0)=\mathtt{d}^*$. From the time reversal property of L\'evy processes,
$\p((X-\underline{X})_\mathbf{ e}=0)=\p(\overline{X}_\mathbf{
e}=0)=\p
(\tau_0^+\ge\mathbf{ e})=\gamma^{-1}$, so that
$\mathtt{d}^*=\gamma^{-1}$.

We will denote by $q_t^*$ and $q_t$ the entrance laws of the reflected
excursions at the maximum and at the minimum, that is, for $t>0$,
\[
q_t(dx)=n(X_t\in dx,t<\zeta)\quad \mbox{and}\quad q_t^*(dx)=n^*(X_t\in
dx,t<\zeta).
\]
They will be considered as measures on $\mathbb{R}_+=[0,\infty)$.
Recall that the law of the lifetime of the reflected excursions is
related to the L\'evy measure
of the ladder time processes, through the equalities
%
%
\begin{eqnarray}\label{pi}
q_t(\mathbb{R_+})&=&n(t<\zeta)=\overline{\pi}(t)+a \quad\mbox{and}
\nonumber
\\[-8pt]
\\[-8pt]
\nonumber
q_t^*(\mathbb{R_+})&=&
n^*(t<\zeta)=\overline{\pi}^*(t)+a^*,
\end{eqnarray}
where $\overline{\pi}(t)=\pi(t,\infty)$ and $\overline{\pi
}^*(t)=\pi
^*(t,\infty)$ and $a$, $a^*$ are the killing rates
of the subordinators $\tau$ and $\tau^*$.

In this paper, we will sometimes write $\mu\ll\nu$, when $\mu$ is
absolutely continuous with respect to $\nu$. We will say that
$\mu$ and $\nu$ are \textit{equivalent} if $\mu\ll\nu$ and $\nu\ll
\mu$.
We will denote by $\lambda$ the Lebesgue measure on $\mathbb{R}$. A
measure which is absolutely continuous
with respect to the Lebesgue measure will sometimes be called {\it
absolutely continuous}. A measure which has no
atoms will be called \textit{continuous}.

\section{Continuity properties of the law of $(g_t,\overline
{X}_t,X_t)$}\label{main}
\setcounter{equation}{0}
In this section, $X$ is any L\'evy process such that $|X|$ is not
subordinator, and except in Theorem~\ref{new}, we assume that
$X$ is not a compound Poisson process.

For $t>0$ and $q>0$, we will denote, respectively by $p_t(dx)$ and
$U_q(dx)$, the semigroup and the resolvent measure
of $X$, that is, for any positive Borel function~$f$,
\begin{eqnarray*}
\e(f(X_t))&=&\int_0^\infty f(x)p_t(dx) \quad\mbox{and}\\
 \int_0^\infty
f(x) U_q(dx)&=&\e\biggl(\int_0^\infty
e^{-qt}f(X_t) \,dt\biggr).
\end{eqnarray*}
Since $U_q(A)=0$ if and only if $\p(X_t\in A)=0$, for $\lambda$ almost
every $t$, it follows that
for all $q$ and $q'$, the resolvent measures $U_q(dx)$ and $U_{q'}(dx)$
are equivalent.
For the same reason, each measure $U_q$ is equivalent to the potential
measure $U_0(dx)=\int_0^\infty\p(X_t\in dx) \,dt$.
In what follows, when comparing the law of $\overline{X}_t$ to the
measures $U_q$, $q\ge0$,
we will take $U(dx)\eqdef U_1(dx)$ as a reference measure.
We will say that a L\'evy process $X$ is of:
\begin{itemize}
\item type 1 if 0 is regular for both $(-\infty,0)$ and $(0,\infty)$;

\item type 2 if 0 is not regular for $(-\infty,0)$;

\item type 3 if 0 is not regular for $(0,\infty)$.
\end{itemize}
We emphasize that since $X$ is not a compound Poisson
process, types 1, 2 and 3
define three exhaustive cases.
Recall that $\mathbb{R}_+=[0,\infty)$, and let $\mathcal{ B}_{\mathbb
{R}_+}$ be the Borel $\sigma$-field on $\mathbb{R}_+$.
For $t>0$, let $\mu_t^+$ be the restriction to
$(\mathbb{R}_+,\mathcal{ B}_{\mathbb{R}_+})$ of the average occupation
measure of $X$, on the time interval $[0,t)$, that is,
\[
\int_{[0,\infty)}f(x) \mu_t^+(dx)=\e\biggl(\int_0^tf(X_s) \,ds\biggr),
\]
for every nonnegative Borel function $f$ on $(\mathbb{R},\mathcal{
B}_{\mathbb{R}})$, such that $f\equiv0$ on $(-\infty,0)$.
Moreover, we will denote by $p_t^+(dx)$ the restriction
of the semigroup $p_t(dx)$ to $(\mathbb{R}_+,\mathcal{ B}_{\mathbb{R}_+})$.
In particular, we have $\mu_t^+=\int_0^t p_s^+ \,ds$.
The law of $\overline{X}_t$ will be considered as a measure on
$(\mathbb
{R}_+,\mathcal{ B}_{\mathbb{R}_+})$. In all the remainder of
this article, we assume that the time $t$ is deterministic and finite.
\begin{theorem}\label{type} For $t>0$, the law of the past supremum
$\overline{X}_t$ can be compared
to the occupation measure $\mu_t^+$ as follows:
\begin{enumerate}[(1)]
\item[(1)] If $X$ is of type $1$, then for all $t>0$, the law of
$\overline{X}_t$ is equivalent to $\mu_t^+$.
\item[(2)] If $X$ is of type $2$, then for all $t>0$, the law of
$\overline{X}_t$ is equivalent
to $p_t^+(dx)+\mu_t^+(dx)$.
\item[(3)] If $X$ is of type $3$, then for all $t>0$, the law of
$\overline{X}_t$ has an atom at $0$ and its restriction
to the open halfline $(0,\infty)$ is equivalent to the restriction of
the measure $\mu_t^+(dx)$ to $(0,\infty)$.
\end{enumerate}
\end{theorem}
It appears clearly from this theorem that the law of $\overline{X}_t$
is absolutely continuous for all $t>0$, whenever 0 is regular
for $(0,\infty)$ and $p_t$ is absolutely continuous, for all $t>0$. We
will see in Theorem~\ref{coro2} that a stronger result actually
holds. Let $U^+(dx)$ be the restriction to $(\mathbb{R}_+,\mathcal{
B}_{\mathbb{R}_+})$ of the resolvent measure $U(dx)$.
Since $\mu_t^+$ is absolutely continuous with respect to $U^+$ for all
$t>0$, the law of the past supremum before $t$ can be
compared to $U^+$ as follows.\looseness=-1
\begin{corollary}\label{lebesgue} Under the same assumptions as in
Theorem~\ref{type}:
\begin{enumerate}[(1)]
\item[(1)] If $X$ is of type $1$, then for any $t>0$, the law of
$\overline{X}_t$ is absolutely
continuous with respect to the resolvent measure $U^+(dx)$.
\item[(2)] If $X$ is of type $2$, then for any $t>0$, the law of
$\overline{X}_t$ is absolutely
continuous with respect to the measure $p_t^+(dx)+U^+(dx)$.
\item[(3)] If $X$ is of type $3$, then the same conclusions as in
$1.$ hold for the measures
restricted to $(0,\infty)$.
\end{enumerate}
\end{corollary}
Whenever $X$ is not a compound Poisson process, the
resolvent measure $U^+(dx)$ is continuous;
see Proposition I.$15$ in~\cite{be}.
Moreover, the measure $p_t^+(dx)$ is also continuous for all $t>0$; see
Theorem 27.4 in Sato~\cite{sa}.
Hence from Corollary~\ref{lebesgue}, for all $t>0$, when $X$ is of type
$1$ or $2$, the law of $\overline{X}_t$ is continuous,
and when it is of type 3, this law has only one atom at $0$. This fact
has already been observed in~\cite{pr}, Lemma~1.

It is known that for a L\'evy process $X$, the law of $X_t$ may be
absolutely continuous for all $t>t_0$, whereas it is continuous
singular for $t\in(0,t_0)$; see Theorem~27.23 and Remark 27.24 in
\cite
{sa}. The following theorem shows that when $X$ is of type 1,
this phenomenon cannot happen for the law of the supremum, that is,
either absolute continuity of the law of $\overline{X}_t$ holds at any
time $t$ or it never holds. We denote by $V(dt,dx)$ the potential
measure of the ladder process $(\tau,H)$ and by $V(dx)$ the potential\vadjust{\goodbreak}
measure of the ladder height process $H$, that is,
\[
V(dt,dx)=\int_0^\infty\p(\tau_s\in dt,H_s\in dx) \,ds\quad \mbox{and}\quad
V(dx)=\int_0^\infty\p(H_s\in dx) \,ds.
\]
Then let $\lambda^+$ be the Lebesgue measure on
$\mathbb{R}_+$.\vspace*{-3pt}

\begin{theorem}\label{th2} Suppose that $X$ is of type $1$. The
following assertions are equivalent:
\begin{enumerate}[(1)]
\item[(1)] The law of $\overline{X}_t$ is absolutely continuous with
respect to $\lambda^+$, for all $t>0$.
\item[(2)] The law of $\overline{X}_t$ is absolutely continuous with
respect to $\lambda^+$, for some $t>0$.
\item[(3)] The resolvent measure $U^+(dx)$ is absolutely continuous
with respect to~$\lambda^+$.
\item[(4)] The resolvent measure $U(dx)$ is absolutely continuous
with respect to $\lambda$.
\item[(5)] The potential measure $V(dx)$ is absolutely continuous
with respect to~$\lambda^+$.
\end{enumerate}
Moreover assertions 1--5 are equivalent to the same assertions
formulated for the dual process $-X$.
In particular, 1--5 hold if and only if the law of $-\underline
{X}_t$ is absolutely continuous with respect
to $\lambda^+$, for all $t>0$.\vspace*{-3pt}
\end{theorem}
Condition 4 of the above theorem is satisfied whenever the
drift coefficient
of the subordinator $H$ is positive; see Theorem II.16 and Corollary
II.20 in~\cite{be}.
Let us also mention that necessary and sufficient conditions for
$U(dx)$ to be absolutely continuous may be found in Theorem 41.15 of
\cite{sa}, and in Proposition 10, Chapter~I
of~\cite{be}. Formally, $U\ll\lambda$ if and only if for some $q>0$
and for all bounded Borel function $f$, the function
$x\mapsto\e_x(\int_0^\infty f(X_t)e^{-qt} \,dt)$ is continuous.
However, we do not know any necessary and sufficient conditions bearing
directly on the characteristic exponent $\psi$ of $X$.
Let us simply recall the following sufficient condition. From Theorem
II.16 in~\cite{be}, if
%
%
\begin{equation}\label{intest}
\int_{-\infty}^\infty\Re\biggl(\frac1{1+\psi(x)}\biggr) \,dx<\infty,
\end{equation}
then $U(dx)\ll\lambda$, with a bounded density. Therefore, if $X$ is of
type 1, then from Theorem~\ref{th2}, condition $(\ref{intest})$
implies that both the laws of $\underline{X}_{\hspace*{1pt}t}$ and $\overline{X}_t$
are absolutely continuous for all $t>0$.

A famous result from~\cite{fu} asserts that when $X$ is a symmetric
process, condition $U\ll\lambda$ implies that
$p_t\ll\lambda$, for all $t>0$. Then it follows from Theorem \ref
{th2} that in this particular case, absolute continuity
of the law of $\overline{X}_t$, for some $t>0$ (hence for all $t>0$)
is equivalent to the absolute continuity of the semigroup
$p_t$, for all $t>0$.\vspace*{-3pt}

\begin{theorem}\label{coro2} If $0$ is regular for $(0,\infty)$, then
the following assertions are equivalent:
\begin{enumerate}[(1)]
\item[(1)] The measures $p_t^+$ are absolutely continuous with
respect to $\lambda^+$, for all $t>0$.\vadjust{\goodbreak}
\item[(2)] The measures $p_t$ are absolutely continuous with respect
to $\lambda$, for all $t>0$.
\item[(3)] The potential measure $V(dt,dx)$ is absolutely continuous
with respect to the Lebesgue measure on $\mathbb{R}_+^2$.
\end{enumerate}
If moreover $X$ is of type $1$, then each of the following assertions
is equivalent to 1--3:
\begin{enumerate}[(1)]
\item[(4)] The law of $(g_t,\overline{X}_t)$ is absolutely continuous
with respect to the Lebesgue measure on
$[0,t]\times\mathbb{R}_+$, for all $t>0$.
\item[(5)] The law of $(g_t,\overline{X}_t,X_t)$ is absolutely
continuous with respect to the Lebesgue measure on
$[0,t]\times\mathbb{R}_+\times\mathbb{R}$, for all $t>0$.
\end{enumerate}
\end{theorem}
We may wonder if the equivalence between assertions (1) and (2)
of Theorem~\ref{coro2} still holds when $t$ is fixed, that is, when 0
is regular for $(0,\infty)$, does the condition $p_t^+\ll\lambda^+$,
imply that $p_t\ll\lambda$?
A counterexample in the case where 0 is not regular for $(0,\infty)$
may easily be found. Take for instance, $X_t=Y_t-S_t$, where $Y$ is a compound
Poisson process with absolutely continuous L\'evy measure, and $S$ is a
subordinator independent of $Y$, whose law at time
$t>0$ is continuous singular. Then clearly $p_t^+\ll\lambda^+$, and
there exists a Borel set $A\subset(-\infty,0)$, such that $\lambda(A)=0$
and $\p(-S_t\in A)>0$, so that $p_t(A)>\p(Y_t=0)\p(S_t\in A)>0$.

Let $Y$ be a c\`adl\`ag stochastic process such that $Y_0=0$, a.s. We
say that $Y$ is an \textit{elementary process} if there is
an increasing sequence $(T_n)$ of nonnegative random variables, such
that $T_0=0$ and $\lim_{n\rightarrow+\infty}T_n=+\infty$, a.s.
and two sequences of finite real-valued random variables $(a_n, n\ge
0)$ and $(b_n, n\ge0)$ such that $b_0=0$ and
%
%
\begin{equation}\label{elem}
Y_t=a_nt+b_n\qquad \mbox{if } t\in[T_n,T_{n+1}).
\end{equation}
We say that $Y$ is a \textit{step process} if it is an elementary process
with $a_n=0$, for all $n$ in the above definition.

\begin{proposition}\label{coro3} Suppose that $0$ is regular for
$(0,\infty)$.
\begin{enumerate}[(1)]
\item[(1)] If $0$ is regular for $(-\infty,0)$, and if the law of
$\overline{X}_t$ is absolutely continuous for some $t>0$,
then for any step process $Y$ which is
independent of $X$, the law of $\sup_{s\le t}(X+Y)_s$ is absolutely
continuous for all $t>0$.
\item[(2)] If $p_t^+\ll\lambda^+$, for all $t>0$, or if $X$ has
unbounded variation, and if at least one of the ladder height processes
$H$ and $H^*$ has a positive drift, then for any elementary stochastic
process $Y$ which is independent of $X$,
the law of $\sup_{s\le t}(X+Y)_s$ is absolutely continuous for all $t>0$.
\end{enumerate}
\end{proposition}

Sufficient conditions for the absolute continuity of the
semigroup may be found in Chapter 5 of~\cite{sa} and in
Section 5 of~\cite{ka}. In particular if $\Pi(\mathbb{R})=\infty$ and
$\Pi\ll\lambda$, then $p_t\ll\lambda$ for all $t>0$.
Proposition 20 in Bouleau and Denis~\cite{bd} asserts that under a
slight reinforcement of this condition, for any independent
c\`adl\`ag process $Y$, the law of $\sup_{s\le t}(X+Y)_s$ is absolutely
continuous, provided it has no atom at 0. In the
particular case where $Y$ is an elementary process, this result is a
consequence of part 2 of Proposition~\ref{coro3}.

In view of Theorems~\ref{th2} and~\ref{coro2}, it is natural to look
for instances of L\'evy processes of type 1 such that
the law of $\overline{X}_t$ is absolutely continuous, whereas
$p_t(dx)$ is not, as well as instances of L\'evy processes
of type 1 such that the law of $\overline{X_t}$ is not absolutely continuous.
The following corollary is inspired from Orey's example~\cite{or}; see
also~\cite{sa}, Exercise 29.12 and Example~41.23.

\begin{corollary}\label{counterexample} Let $X$ be a L\'evy process
whose characteristic exponent~$\psi$, that is, $\e(e^{i\lambda
X_t})=e^{-t\psi(\lambda)}$
is given by
\[
\psi(\lambda)=\int_{\mathbb{R}}\bigl(1-e^{i\lambda x}+i\lambda x\ind_{\{
|x|<1\}}\bigr)\Pi(dx).
\]
Let $\alpha\in(1,2)$ and $c$ be an integer such that $c>2/(2-\alpha)$,
and set $a_n=2^{-c^n}$.
\begin{enumerate}[(1)]
\item[(1)] If $\Pi(dx)=\sum_{n=1}^\infty a_n^{-\alpha}\delta
_{-a_n}(dx)$, then $X$ is of type $1$, and
for all $t>0$, the law of
$\overline{X}_t$ is absolutely continuous, whereas $p_t(dx)$ is
continuous singular.
\item[(2)] If $\Pi(dx)=\sum_{n=1}^\infty a_n^{-\alpha}(\delta
_{-a_n}(dx)+\delta_{a_n}(dx))$, then $X$ is of type $1$, and
for all $t>0$, the law $\overline{X}_t$ is not absolutely continuous.
\end{enumerate}
\end{corollary}

We end this section with the case of compound Poisson processes. Recall
that any such process can be expressed as:
\[
X_t=S_{N_t},\qquad t\ge0,
\]
where $S_0=0$, $S_n=\sum_{k=1}^nX_k$, $n\ge1$, $(X_k)_{k\ge1}$ are
i.i.d.~random variables, and $(N_t)_{t\ge0}$ is a Poisson process
with any intensity, which is independent from the sequence $(X_k)_{k\ge
1}$. We keep the same notation for the measures
$p_t^+$, $\mu_t^+$ and $U^+$, which are defined with respect to $X$, as
before. We denote by $\upsilon^+$ the restriction to
$(\mathbb{R}_+,\mathcal{ B}_{\mathbb{R}_+})$ of the potential
measure of the random walk $(S_n)_{n\ge0}$, that is,
\[
\upsilon^+(A)=\sum_{n=0}^\infty\p(S_n\in A),\qquad A\in\mathcal{
B}_{\mathbb{R}_+}.
\]

\begin{theorem}\label{new} Let $X$ be a compound Poisson process. Then
for all $t>0$, the measures
\[
\p(\overline{X}_t\in dx),\qquad p_t^+(dx),\qquad \mu_t^+(dx),\qquad
U^+(dx)\quad \mbox{and}\quad \upsilon^+(dx)
\]
are equivalent.
\end{theorem}

As a consequence, when $X$ is a compound Poisson process, for
any $t>0$ and $t'>0$,
the laws of $\overline{X}_t$ and $\overline{X}_{t'}$ are equivalent.
This question is still open in the general case.

\section{An expression for the joint law of $(g_t,\overline
{X}_t,X_t)$}\label{expressions}

In this section, we assume that $|X|$ is not subordinator and that $X$
is not a compound Poisson process.

The following theorem presents a path decomposition of the L\'evy
process $X$, over the interval $[0,t]$, at time $g_t$.
More specifically, it states that conditionally on $g_t=s$, the
returned pre-$g_t$ part and the post-$g_t$ part are
distributed according to the laws $n^*( \cdot| s<\zeta)$ and $n(
\cdot| t-s<\zeta)$, respectively.
Actually, we will essentially focus on its corollaries
which provide some representations of the joint law of $g_t$,
$\overline{X}_t$ and $X_t$, at a fixed time $t$, in terms of the
entrance laws $(q_s)$ and $(q^*_s)$. Besides
they will be applied in Section~\ref{proofs} for the proofs of the
results of Section~\ref{main}.

For $\omega\in\mathcal{ D}$ and $s\ge0$, we set $\Delta^\pm_s(\omega
)=(\omega_s-\omega_{s-})^\pm$,
where $\omega_{0-}=\omega_0$. Then we define the (special) \textit{shift}
operator by
\[
\theta_s(\omega)=\bigl(\omega_{s-}-\omega_{s+u}+\Delta_s^+(\omega),
u\ge0\bigr).
\]
The \textit{killing} operator and the \textit{return} operator are
respectively defined as
follows:
\begin{eqnarray*}
k_s(\omega)&=&\cases{
\omega_u,&\quad $0\le u<s,$\vspace*{2pt}\cr
\omega_s,& \quad $u\ge s$,}
\\
r_s(\omega)&=&\cases{
\omega_s-\omega_{(s-u)-}-\Delta_s^-(\omega),& \quad $0\le u<s,$\vspace*{2pt}\cr
\omega_s-\omega_{0}-\Delta_s^-(\omega),&\quad $u\ge s.$}
\end{eqnarray*}
We also denote by $\omega^0$ the path which is identically equal to 0.
\begin{theorem}\label{mainth}
Fix $t>0$, let $f$ be any bounded Borel function and\break let $F$ and~$K$ be
any bounded Borel functionals
which are defined on the space~$\mathcal{D}$.\looseness=1
\begin{enumerate}[(1)]
\item[(1)] If $X$ is of type $1$, then
%
%
\begin{eqnarray}\label{1a}
&&\e\bigl(f(g_t)\cdot F\circ r_{g_t}\cdot K\circ k_{t-g_t}\circ\theta
_{g_t}\bigr)
\nonumber
\\[-8pt]
\\[-8pt]
\nonumber
&&\qquad =
 \int_0^tf(s)n^*(F\circ k_s,s<\zeta)n(K\circ
k_{t-s},t-s<\zeta) \,ds.
\end{eqnarray}
\item[(2)] If $X$ is of type $2$, then
%
%
\begin{eqnarray}\label{2a} 
&&\e\bigl(f(g_t)\cdot F\circ r_{g_t}\cdot K\circ k_{t-g_t}\circ\theta
_{g_t}\bigr)\nonumber\\
&&\qquad=
 \int_0^tf(s)n^*(F\circ k_s,s<\zeta)n(K\circ
k_{t-s},t-s<\zeta) \,ds\\
&&\qquad\quad{} +\mathtt{d} f(t)n^*(F\circ k_t,t<\zeta)K(\omega
^0).\nonumber
\end{eqnarray}
\item[(3)] If $X$ is of type $3$, then
%
%
\begin{eqnarray}\label{3a} 
&&\e\bigl(f(g_t)\cdot F\circ r_{g_t}\cdot K\circ k_{t-g_t}\circ\theta
_{g_t}\bigr)\nonumber\\
&&\qquad =
 \int_0^tf(s)n^*(F\circ k_s,s<\zeta)n(K\circ
k_{t-s},t-s<\zeta) \,ds\\
&&\qquad\quad{} +\mathtt{d}^*f(0)F(\omega^0)n(K\circ k_t,t<\zeta
).\nonumber
\end{eqnarray}
\end{enumerate}
\end{theorem}

Simultaneously to our work, a similar path decomposition has
been obtained in~\cite{ya}, when $X$ is of type 1.
In the later work, the post-$g_t$ part of $(X_s, 0\le s\le t)$ is
expressed in terms of the law of the meander, that is, $\mathbb
{M}^{(t)}=n( \cdot| t<\zeta)$; see Theorem 5.1.

By applying Theorem~\ref{mainth} to the joint law of $g_t$, together
with the terminal values of the pre-$g_t$
and the post-$g_t$ parts of $(X_s, 0\le s\le t)$, we obtain the
following representation for the law of the triple:
$(g_t,\overline{X}_t,X_t)$. Moreover, when $\lim_{t\rightarrow\infty
}X_t=-\infty$, a.s., we define $\overline{X}_\infty=\sup_tX_t$,
the overall supremum of $X$ and $g_\infty=\sup\{t:X_t=\overline
{X}_\infty\mbox{ or } X_{t-}=\overline{X}_\infty\}$,
the location of this supremum. Then we obtain
the same kind of representation for $(g_\infty,\overline{X}_\infty)$.
We emphasize that in the next result, as well as in Corollaries
\ref{law1} and~\ref{semigroup}, at least one of the drift coefficients
$\mathtt{d}$ and $\mathtt{d}^*$ is zero.
\begin{theorem}\label{law}
The law of $(g_t,\overline{X}_t,X_t)$ fulfills the following representation:
%
%
\begin{eqnarray}\label{both}
&&\p(g_t\in ds,\overline{X}_t\in dx,\overline{X}_t-X_t\in
dy)
\nonumber
\\
&&\qquad=q_s^*(dx)q_{t-s}(dy)\ind_{[0,t]}(s) \,ds
 +\mathtt{d} \delta_{\{t\}}(ds)q_t^*(dx)\delta_{\{0\}
}(dy)\\
&&\quad\qquad{} +\mathtt{d}^*\delta_{\{0\}}(ds)\delta_{\{0\}
}(dx)q_t(dy).\nonumber
\end{eqnarray}
If moreover $\lim_{t\rightarrow\infty}X_t=-\infty$, a.s., then
%
%
\begin{equation}\label{infty}
\p(g_\infty\in ds,\overline{X}_\infty\in dx)=a q_s^*(dx) \,ds+\mathtt{d}^*a\delta_{\{(0,0)\}}(ds,dx),
\end{equation}
where $a$ is the killing rate of the ladder time process $\tau$.
\end{theorem}

We derive from Theorem~\ref{law} that when $X$ is of type 1,
the law of the time $g_t$ is equivalent to
the Lebesgue measure, with density $s\mapsto n^*(s<\zeta)n(t-s<\zeta
)\ind_{[0,t]}(s)$.
This theorem illustrates the importance of the entrance laws $q_t$ and
$q_t^*$ for the
computation of some distributions involved in fluctuation theory. We
give below a couple of examples where some explicit forms
can be obtained for $q_t$, $q_t^*$ and the law of $(g_t,\overline
{X}_t,X_t)$. When $q_t(dx)\ll\lambda^+$ [resp., $q^*_t(dx)\ll\lambda^+$],
we will denote by $q_t(x)$
[resp., $q_t^*(x)$] the density of $q_t(dx)$ [resp., $q_t^*(dx)$].

\begin{example}\label{ex1} Suppose that $X$ is a Brownian motion with
drift, that is, $X_t=B_t+ct$, where $B$ is the standard Brownian motion\vadjust{\goodbreak}
and $c\in\mathbb{R}$. We derive for instance from
Lemma~\ref{equivalence1} in Section~\ref{proofs} that
\[
q_t(dx)=\frac x{\sqrt{\pi t^3}} e^{-(x-c)^2/2t} \,dx \quad\mbox{and}\quad
q_t^*(dx)=\frac x{\sqrt{\pi t^3}} e^{-(x+c)^2/2t} \,dx.
\]
Then expression (\ref{both}) in Theorem~\ref{law} allows us to compute
the law of the triple $(g_t,\overline{X}_t,X_t)$.
\end{example}

\begin{example}\label{ex2} Recently, the density of the measure
$q_t(dx)$ for the symmetric Cauchy process
has been computed in~\cite{ma}.
\begin{eqnarray*}
q_t(x)&=&q^*_t(x)\\
&=&\sqrt{2}\frac{\sin(\pi/8+3/2\arctan
( x/t))}{(t^2+x^2)^{3/4}}\\
&&{} -\frac1{2\pi}\int_0^\infty\frac{y}{(1+y^2)(xy+t)^{3/2}}
\exp\biggl(-\frac1\pi\int_0^\infty\frac{\log(y+s)}{1+s^2} \,ds\biggr)
\,dy.
\end{eqnarray*}
As far as we know, this example, together with the case of Brownian
motion with drift (Example~\ref{ex1}), are the only instances of L\'evy
processes where
the measures $q_t(dx)$, $q^*_t(dx)$ and the law of the triplet
$(g_t,\overline{X}_t,X_t)$ can be computed explicitly.
\end{example}

\begin{example}\label{ex3} Recall from (\ref{pi}) that $q_t(\mathbb
{R}_+)=n(t<\zeta)$ and $q_t^*(\mathbb{R}_+)=n^*(t<\zeta)$, so that
we can derive from Theorem~\ref{law}, all possible marginal laws in the
triplet $(g_t,\overline{X}_t,X_t)$. In particular, when $X$ is stable,
the ladder time process $\tau$ also satisfies the scaling property with
index $\rho=\p(X_1\ge0)$, so we derive from the normalization
$\kappa(1,0)=1$ and (\ref{pi}) that $n(t<\zeta)=t^{-\rho}/\Gamma
(1-\rho
)$. Moreover $q_t^*$ and $q_t$ are absolutely
continuous in this case (it can be derived, e.g., from part 4 of
Lemma~\ref{equivalence1} in the next section).
Then a consequence of (\ref{both}) is the following form of the joint
law of $(g_t,\overline{X}_t)$:
%
%
\begin{equation}\label{lawstable}
\p(g_t\in ds,\overline{X}_t\in dx)=\frac{(t-s)^{-\rho}}{\Gamma
(1-\rho
)}\ind_{[0,t]}(s) q_s^*(x) \,ds \,dx.
\end{equation}
Note that this computation is implicit in~\cite{ac}; see Corollary 3
and Theorem~5. A~more explicit form is given in (\ref{lawsn}),
after Proposition~\ref{plsn}, in the case where the process has no
positive jumps. Note also that when $X$ is stable,
the densities $q_t$ and $q_t^*$ satisfy the scaling properties
\[
q_t(y)=t^{-\rho-1/\alpha}q_1(t^{-1/\alpha}y) \quad\mbox{and}\quad
q_t^*(x)=t^{\rho-1-1/\alpha}q_1^*(t^{-1/\alpha}x).
\]
These properties together with Theorem~\ref{law} imply that the three
r.v.s $g_t$,
$\overline{X}_t/g_t^{1/\alpha}$ and $(\overline
{X}_t-X_t)/(t-g_t)^{1/\alpha}$ are independent and have densities
\[
\frac{\sin(\pi\rho)}{\pi}s^{\rho-1}(t-s)^{-\rho}\ind_{[0,t]}(s),\qquad
\Gamma(\rho)q^*_1(x) \quad\mbox{and}\quad \Gamma(1-\rho)q_1(y),
\]
respectively. The independence between $g_t$, $\overline
{X}_t/g_t^{1/\alpha}$ and $(\overline{X}_t-X_t)/\break(t-g_t)^{1/\alpha}$
has recently been proved in Proposition 2.39 of~\cite{co}.
\end{example}

It is clear that an expression for the law of $\overline
{X}_t$ follows directly from Theorem~\ref{law} by integrating (\ref{both})
over $s$ and $y$. However, for convenience in the proofs of
Section~\ref{proofs}, we write it here
in a proper statement. An equivalent version of Corollary~\ref{law1}
may also be found in~\cite{ds}, Lemma 6.\vspace*{-3pt}

\begin{corollary}\label{law1} The law of $\overline{X}_t$ fulfills the
following representation:
%
%
\begin{equation}\label{both1}
\p(\overline{X}_t\in dx)=\int_0^tn(t-s<\zeta)q_s^*(dx) \,ds+\mathtt{d}
q_t^*(dx)+\mathtt{d}^*n(t<\zeta)\delta_{\{0\}}(dx).\hspace*{-35pt}\vspace*{-3pt}
\end{equation}
\end{corollary}

Another remarkable, and later useful, direct consequence of
Theorem~\ref{law} is the following
representation of the semigroup of $X$ in terms of the entrance laws
$(q_s)$ and $(q_s^*)$.\vspace*{-3pt}

\begin{corollary}\label{semigroup} Let us denote the measure $q_t(-dx)$
by $\overline{q}_t(dx)$. We extend the measures
$\overline{q}_t(dx)$ and $q_t^*(dx)$ to $\mathbb{R}$ by setting
$\overline{q}_t(A)=\overline{q}_t(A\cap\mathbb{R}_-)$ and
$q_t^*(A)=q_t^*(A\cap\mathbb{R}_+)$, for any Borel set $A\subset
\mathbb{R}$. Then we have the following identity between measures
on $\mathbb{R}$:
%
%
\begin{equation}\label{sgent}
p_t=\int_0^t\overline{q}_s*q_{t-s}^* \,ds+\mathtt{d} q_t^*+\mathtt{d}^*\overline{q}_t.\vspace*{-3pt}
\end{equation}
\end{corollary}

Now we turn to the particular case where $X$ has no positive jumps.
Then, 0 is always regular for $(0,\infty)$.
When moreover 0 is regular for $(-\infty,0)$, since $H_t\equiv t$, it
follows from Theorem~\ref{th2} and the remark thereafter that
the law of $\overline{X}_t$ is absolutely continuous. In the next
result, we present
an explicit form of its density. We set $c=\Phi(1)$, where $\Phi$ is
the Laplace exponent of
the first passage process $T_x=\inf\{t:X_t>x\}$, which in this case, is
related to the ladder time process by $T_x=\tau_{cx}$.\vspace*{-3pt}

\begin{proposition}\label{plsn} Suppose that the L\'evy process $X$ has
no positive jumps.
\begin{enumerate}[(1)]
\item[(1)] If $0$ is regular for $(-\infty,0)$, then for $t>0$, the
couple $(g_t,\overline{X}_t)$ has law
%
%
\begin{eqnarray}
\p(g_t\in ds,\overline{X}_t\in dx)&=&cxp_s^+(dx)n(t-s<\zeta
)s^{-1}\ind
_{(0,t]}(s) \,ds\label{0528}\\
&=&cn(t-s<\zeta)\ind_{(0,t]}(s)\p(\tau_{cx}\in ds) \,dx.\label{7520}
\end{eqnarray}
In particular, the density of the law of $\overline{X}_t$ is given by
the function
\[
x\mapsto\int_0^tcn(t-s<\zeta)\p(\tau_{cx}\in ds).
\]
\item[(2)] If $0$ is not regular for $(-\infty,0)$, then for all $t>0$,
%
%
\begin{eqnarray}\label{1528}
\p(g_t\in ds,\overline{X}_t\in dx)&=&cxn(t-s<\zeta)s^{-1}\ind
_{(0,t]}(s)p^+_s(dx) \,ds
\nonumber
\\[-8pt]
\\[-8pt]
\nonumber
&&{}+
\mathtt{d}cxt^{-1}p_t^+(dx)\delta_{\{t\}}(ds).
\end{eqnarray}
Moreover, we have the following identity between measures on $[0,\infty)^3$:
%
%
\begin{eqnarray}\label{7620}
\p(g_t\in ds,\overline{X}_t\in dx) \,dt&=&cn(t-s<\zeta)\ind
_{(0,t]}(s)\p
(\tau_{cx}\in ds) \,dx \,dt
\nonumber
\\[-8pt]
\\[-8pt]
\nonumber
&&{}+\mathtt{d}c\p(\tau_{cx}\in \,dt)\delta_{\{t\}}(ds) \,dx.
\end{eqnarray}
\end{enumerate}
\end{proposition}

\begin{example} Using the series development (14.30), page
88 in~\cite{sa} for $p_s^+(dx)$,
we derive from (\ref{0528}) in Proposition~\ref{plsn}, the following
reinforcement of expression (\ref{lawstable}). When $X$
is stable and spectrally negative, the density of $(g_t,\overline
{X}_t)$ is given by
%
%
\begin{eqnarray}\label{lawsn}
\qquad\frac{c}{\pi\Gamma({(\alpha-1)}/\alpha)(t-s)^{1/\alpha
}}\sum_{n=1}^\infty
\frac{\Gamma(1+n/\alpha)}{n!}\sin\biggl(\frac{\pi n}\alpha
\biggr)s^{-{(n+\alpha)}/\alpha}x^n,
\nonumber
\\[-8pt]
\\[-8pt]
\eqntext{s\in[0,t], x\ge0,}
\end{eqnarray}
which completes Proposition 1, page 282 in~\cite{bi}.
\end{example}

We end this section with a remark on the existence of a
density with respect to the Lebesgue measure, for the law of the local
time of general Markov processes. From (\ref{7620}), we derive that
$\p
(\tau_x\ge t) \,dt=\int_0^x\int_{(0,t]}n^*(t-s<\zeta)\p(\tau_y\in ds)
\,dy \,dt+
\mathtt{d}\p(\tau_x\in dt)$.
Actually, this identity may be generalized to any subordinator $S$ with
drift $b$, killing rate $k$ and
L\'evy measure $\nu$. Set $\bar{\nu}(t)=\nu(t,\infty)+k$, then the
characteristic exponent $\Phi$ of $S$
is given by
\[
\Phi(\alpha)=\alpha b+\alpha\int_0^\infty e^{-\alpha t}\bar{\nu}(t)
\,dt,
\]
from which and Fubini theorem, we derive that for all $x\ge0$ and
$\alpha>0$,
\begin{eqnarray*}
\frac1\alpha\e(1-e^{-\alpha S_x})&=&\biggl(b+\int_0^\infty e^{-\alpha
t}\bar{\nu}(t) \,dt\biggr)\frac{\e(1-e^{-\alpha S_x})}{\Phi(\alpha)},\\
\int_0^\infty e^{-\alpha t}\p(S_x>t) \,dt&=&\biggl(b+\int_0^\infty
e^{-\alpha t}\bar{\nu}(t) \,dt\biggr)\int_0^\infty e^{-\alpha t}
\int_0^x\p(S_y\in dt) \,dy.
\end{eqnarray*}
Inverting the Laplace transforms on both sides of this identity gives
for all $x\ge0$, the following identity between measures:
\[
\p(S_x>t) \,dt=\int_0^x\int_{(0,t]}\bar{\nu}(t-s) \p(S_y\in ds) \,dy
\,dt+b\int_0^x\p(S_y\in dt) \,dy.
\]
In particular, if $S$ has no drift coefficient, then the law of
$L_t\eqdef\inf\{u:S_u>t\}$ has density
\[
\frac{\p(L_t\in dx)}{dx}=\int_{(0,t]}\bar{\nu}(t-s) \p(S_x\in ds).
\]
This computation shows that if $a\in\mathbb{R}$ is a regular state of
any real Markov process~$M$ such that $\int_0^t\ind_{\{M_s=a\}} \,ds=0$,
a.s. for all $t$, then the law of the local time of $M$, at level $a$,
is absolutely continuous,\vadjust{\goodbreak} for any time $t>0$.
This last result is actually a particular case of~\cite{dr}, where it
is proved that for any non creeping L\'evy process,
the law of the first passage time over $x>0$ is always absolutely
continuous.

\section{Proofs and further results}\label{proofs}
We first prove Theorems~\ref{mainth} and~\ref{law}, since they will be
used in the proofs of the results of Section~\ref{main}.

\begin{pf*}{Proof of Theorem \protect\ref{mainth}}
Let $\mathbf{e}$ be an
exponential time with parameter $q>0$ which is
independent of $(X,\mathbb{P})$. Recall the notations of Section \ref
{prelim}, and for $\omega\in\mathcal{ D}$, define
$d_s(\omega)=\inf\{u>s:\omega_u=0\}$, so that $d_s(\overline{X}-X)$
corresponds to the right extremity of the excursion
of $\overline{X}-X$, which straddles the time~$s$. From the
independence of $\mathbf{e}$ and Fubini theorem, we have for all
bounded function $f$ on $\mathbb{R}_+$ and for all bounded Borel
functionals $F$ and $K$ on~$\mathcal{ D}$,
\begin{eqnarray*}
&&\e\bigl(f(g_{\mathbf{e}})F\circ r_{g_{\mathbf{e}}}K\circ k_{{\mathbf
{e}}-g_{\mathbf{e}}}\circ\theta_{g_{\mathbf{e}}}\bigr)\nonumber\\
&&\qquad=
\e\biggl(\int_0^\infty q e^{-q t} f(g_{t})F\circ r_{g_{t}}K\circ
k_{t-g_{t}}\circ\theta_{g_{t}} \,dt\biggr)
\nonumber
\\[-8pt]
\\[-8pt]
\nonumber
&&\qquad=\e\biggl(\sum_{s\in G}q e^{-q s}f(s)F\circ r_s\int_s^{d_s}e^{-q
(u-s)} K\circ k_{u-s}\circ\theta_{s}
\,du\biggr)\\
&& \qquad\quad{}+\e\biggl(\int_0^\infty q e^{-q
t}f(t)F\circ r_t\ind_{\{g_t=t\}} \,dt\biggr)K(\omega^0).\nonumber
\end{eqnarray*}
Recall from Section~\ref{prelim} that $\varepsilon^s$ denotes the
excursion starting at $s$. Then
%
%
\begin{eqnarray}\label{3618}
&&\e\bigl(f(g_{\mathbf{e}})F\circ r_{g_{\mathbf{e}}}K\circ k_{{\mathbf
{e}}-g_{\mathbf{e}}}\circ\theta_{g_{\mathbf{e}}}\bigr)\nonumber\\
&&\qquad=
\e\biggl(\sum_{s\in G}q e^{-q s}f(s)F\circ r_s\int_0^{d_s-s}e^{-q u}
K(\varepsilon^s\circ k_u)
\,du\biggr)\\
&&\quad\qquad{}+\e\biggl(\int_0^\infty q e^{-q t}f(t)F\circ r_t\ind_{\{X_t=\overline
{X}_t\}} \,dt\biggr)K(\omega^0).\nonumber
\end{eqnarray}
The process
\[
(s,\omega,\varepsilon)\mapsto e^{-q s}f(s)F\circ r_s(\omega)\int
_0^{\zeta
(\varepsilon)} e^{-q u}K\circ k_u(\varepsilon) \,du
\]
is $\mathcal{ P}(\mathcal{ F}_s)\otimes\mathcal{ E}$-measurable, so that by
applying (\ref{compensation}) and (\ref{delta}) to equality
(\ref{3618}), we obtain
%
%
\begin{eqnarray}\label{4264}
&&\frac1q\e\bigl(f(g_{\mathbf{e}})F\circ r_{g_{\mathbf{e}}}K\circ
k_{{\mathbf
{e}}-g_{\mathbf{e}}}\circ\theta_{g_{\mathbf{e}}}\bigr)\nonumber\\
&&\qquad=\e\biggl(\int_0^\infty \,dL_s e^{-q s}f(s)F\circ r_s\biggr)n\biggl(\int
_0^\zeta e^{-q u}K\circ k_u \,du\biggr)\\
&&\qquad\quad{} +\mathtt{d} \e\biggl(\int_0^\infty \,dL_s e^{-q s}f(s)F\circ
r_s\biggr)K(\omega^0).\nonumber
\end{eqnarray}
From the time reversal property of L\'evy processes (see Lemma 2, page
45 in~\cite{be})
under $\p$, we have $X\circ k_e\ed X\circ r_e$, so that
%
%
\begin{eqnarray}\label{4265}
\e\bigl(f(g_{\mathbf{e}})F\circ r_{g_{\mathbf{e}}}K\circ k_{{\mathbf
{e}}-g_{\mathbf{e}}}\circ\theta_{g_{\mathbf{e}}}\bigr)=
\e\bigl(f(\mathbf{e}-g^*_{\mathbf{e}})K\circ r_{g^*_{\mathbf{e}}}F\circ
k_{{\mathbf{e}}-g^*_{\mathbf{e}}}
\circ\theta_{g^*_{\mathbf{e}}}\bigr).\hspace*{-35pt}
\end{eqnarray}
Doing the same calculation as in (\ref{4264}) for the reflected process
at its minimum $X-\underline{X}$, we get
%
%
\begin{eqnarray}\label{4268}
&&\frac1q\e\bigl(f(\mathbf{e}-g^*_{\mathbf{e}})K\circ r_{g^*_{\mathbf
{e}}}F\circ k_{{\mathbf{e}}-g^*_{\mathbf{e}}}\circ\theta
_{g^*_{\mathbf
{e}}}\bigr)\nonumber\\[-3pt]
&&\qquad=\e\biggl(\int_0^\infty \,dL_s^* e^{-q s}K\circ r_s\biggr)n^*\biggl(\int
_0^\zeta e^{-q u}f(u)F\circ k_u \,du\biggr)\\[-3pt]
&&\qquad\quad{}
+\mathtt{d}^*
\e\biggl(\int_0^\infty \,dL_s^* e^{-q s}K\circ r_s\biggr)f(0)F(\omega
^0)\nonumber.
\end{eqnarray}
Then we derive from (\ref{4264}), (\ref{4265}) and (\ref{4268}), the
following equality:
%
%
\begin{eqnarray}\label{4369}
&&\e\biggl(\int_0^\infty \,dL_s e^{-q s}f(s)F\circ r_s\biggr)n\biggl(\int
_0^\zeta e^{-q u}K\circ k_u \,du\biggr)\nonumber\\[-3pt]
&&\quad{} +\mathtt{d} \e\biggl(\int_0^\infty \,dL_s e^{-q
s}f(s)F\circ r_s\biggr)K(\omega^0)
\nonumber
\\[-9pt]
\\[-9pt]
\nonumber
&&\qquad=\e\biggl(\int_0^\infty \,dL_s^* e^{-q s}K\circ r_s\biggr)n^*\biggl(\int
_0^\zeta e^{-q u}f(u)F\circ k_u \,du\biggr)\\[-3pt]
&&\qquad\quad{} +\mathtt{d}^* \e\biggl(\int_0^\infty \,dL_s^* e^{-q
s}K\circ r_s\biggr)f(0)F(\omega^0).\nonumber
\end{eqnarray}
Then by taking $f\equiv1$, $F\equiv1$ and $K\equiv1$, we derive from
(\ref{4264}) that
%
%
\begin{equation}\label{3247}
\kappa(q,0)=n(1-e^{-q\zeta})+q\mathtt{d}.
\end{equation}
Now suppose that $X$ is of type 1 or 2, so that $\mathtt{d}^*=0$, from
what has been recalled in Section~\ref{prelim}.
Hence with $K\equiv1$ in (\ref{4369}) and using (\ref{3247}), we have
%
%
\begin{eqnarray}\label{4036}
&&\e\biggl(\int_0^\infty \,dL_s e^{-q s}f(s)F\circ r_s\biggr)\kappa
(q,0)\kappa^*(q,0)
\nonumber
\\[-9pt]
\\[-9pt]
\nonumber
&&\qquad=q n^*\biggl(\int_0^\zeta e^{-q u}f(u)F\circ k_u \,du\biggr).
\end{eqnarray}
But using (\ref{wh}) and plugging (\ref{4036}) into (\ref{4264}) gives
\begin{eqnarray*}
&&\e\biggl(\int_0^\infty e^{-q t} f(g_{t})F\circ r_{g_{t}}K\circ
k_{t-g_{t}}\circ\theta_{g_{t}} \,dt\biggr)\\[-3pt]
&&\qquad=n^*\biggl(\int_0^\zeta e^{-q u}f(u)F\circ k_u \,du\biggr)
n\biggl(\int_0^\zeta e^{-q u}K\circ k_u \,du\biggr)\\[-3pt]
&&\qquad\quad{}+\mathtt{d} n^*\biggl(\int_0^\zeta e^{-q u}f(u)F\circ k_u \,du
\biggr)K(\omega^0),
\end{eqnarray*}
so that identities (\ref{1a}) and (\ref{2a}) follow for $\lambda
$-almost every $t>0$, by inverting the Laplace transforms in this equality.
Then we easily check that for all $t>0$, the functionals $g_t$,
$r_{g_t}$ and $k_{t-g_t}\circ\theta_{g_t}$ are a.s.~continuous, at any
time $t$.
Hence for any bounded and continuous functions $f$, $F$ and $K$, from
Lebesgue's theorem of dominated convergence,
the left-hand sides of (\ref{1a}) and (\ref{2a}) are continuous in $t$.
From the same arguments, the functions
$t\mapsto n^*(F\circ k_t,t<\zeta)$ and $t\mapsto n(K\circ k_k,t<\zeta)$
are continuous. Hence, from general properties of the convolution product,
the right-hand sides of (\ref{1a}) and (\ref{2a}) are continuous, so
that these identities are valid for all $t>0$.
Then we extend this result to any bounded Borel functions $f$, $F$ and~$K$ through a classical
density argument. Finally, (\ref{3a}) is obtained in the same way as
parts 1 and~2.
\end{pf*}

\begin{pf*}{Proof of Theorem \protect\ref{law}}
Let $g$ and $h$ be two bounded Borel functions on $\mathbb{R_+}$, and
define the functionals $K$ and $F$ on $\mathcal{ D}$ by
$F(\omega)=g(\omega_\zeta)$ and $K(\omega)=h(\omega_\zeta)$. Then we
may check that for $\varepsilon\in\mathcal{ E}$ and $t<\zeta(\varepsilon)$,
$F\circ k_t(\varepsilon)=g(\varepsilon_t)$ and $K\circ k_t(\varepsilon
)=h(\varepsilon_t)$. Moreover since the lifetime of the path
$k_{t-g_t}\circ\theta_{g_t}(\omega)$ is $t-g_t(\omega)$ and $\Delta
^+_{g_t}(\omega)=(\omega_{g_t}-\omega_{g_t-})^+=(\overline
{X}_t-\overline{X}_{t-})(\omega)$,
we have $F\circ r_{g_t}\circ X=g(\overline{X}_t)$ and $K\circ
k_{t-g_t}\circ\theta_{g_t}\circ X=h(\overline{X}_t-X_t)$, so that by
applying Theorem~\ref{mainth} to the functionals $F$ and $K$, we obtain
(\ref{both}).

To prove (\ref{infty}), we first note that $\lim_{t\rightarrow\infty
}(g_t,\overline{X}_t)=(g_\infty,\overline{X}_\infty)$, a.s.
Then let $f$ be a bounded and
continuous function which is defined on $\mathbb{R}_+^2$. We have from~(\ref{both}),
\begin{eqnarray*}
\e(f(g_t,\overline{X}_t))&=&\int_0^tf(s,x)n(t-s<\zeta) q_s^*(dx)
\,ds\\
&&{}+\mathtt{d}\int_0^\infty f(t,x)q_t^*(dx)+\mathtt{d}^*n(t<\zeta)f(0,0).
\end{eqnarray*}
On the one hand, we see from (\ref{pi}) that $\lim_{t\rightarrow
\infty}
n(t<\zeta)=n(\zeta=\infty)=a>0$.
On the other hand, $\lim_{t\rightarrow\infty} n^*(t<\zeta)=0$, and
since the term $\mathtt{d}\int_0^\infty f(t,x)q_t^*(dx)$ is bounded by
$Cn^*(t<\zeta)$, where $|f(s,x)|\le C$, for all $s,x$, it converges to
0 as $t$ tends to $\infty$. This allows
us to conclude.
\end{pf*}

Recall that the definition of the ladder height process $(H_t)$ has
been given in Section~\ref{prelim}.
Then define $(\ell_x, x\ge0)$ as the right continuous inverse of $H$,
that is,
\[
\ell_x=\inf\{t:H_t> x\}.
\]
Note that for types 1 and 2, since $H$ is a strictly increasing subordinator,
the process $(\ell_x, x\ge0)$ is continuous, whereas in type 3, since
$H$ is a compound Poisson process, then $\ell$
is a c\`adl\`ag jump process.
Parts 1 and 2 of the following lemma are reinforcements of Theorems 3
and 5 in~\cite{ac}. Part 1 is also stated in Proposition 9 of
\cite{do}. Recall that $V(dt,dx)$ denotes the potential
measure of the ladder process $(\tau,H)$.\vadjust{\goodbreak}

\begin{lemma}\label{equivalence1} Let $X$ be a L\'evy process which is
not a compound Poisson process
and such that $|X|$ is not a subordinator.
\begin{enumerate}[(1)]
\item[(1)] The following identity between measures holds on $\mathbb
{R}_+^3$:
%
\begin{equation}\label{4594}
u\p(X_t\in dx, \ell_x\in du) \,dt=t\p(\tau_u\in dt,H_u\in dx) \,du.
\end{equation}
\item[(2)] The following identity between measures holds on $\mathbb
{R}_+^2$:
%
\begin{equation}\label{26811}
\mathtt{d}^*\delta_{\{(0,0)\}}(dt,dx)+q_t^*(dx) \,dt=V(dt,dx),
\end{equation}
moreover for all $t>0$, and for all Borel sets $B\in\mathcal{ B}_{\mathbb
{R}_+}$, we have
%
%
\begin{equation}\label{2681}
q_t^*(B)=t^{-1}\e\bigl(\ell(X_t)\ind_{\{X_t\in B\}}\bigr).
\end{equation}
\item[(3)] For all $t>0$, the measures $q_t^*(dx)$ and $p_t^+(dx)$
are equivalent on $\mathbb{R}_+$.
\end{enumerate}
\end{lemma}
\begin{pf} When 0 is regular for $(-\infty,0)$, part (1) is proved in
Theorem 3 of~\cite{ac} and when 0 is regular for both $(-\infty,0)$
and $(0,\infty)$, part (2)~is proved in Theorem~5 of~\cite{ac}.
Although the proofs of parts (1) and (2) in the general case follow about the same
scheme as in~\cite{ac}, it is necessary to check some details.\looseness=-1

First recall the so-called Fristedt identity which is established in
all the cases concerned by this
lemma, in~\cite{ky}; see Theorem 6.16. For all $\alpha\ge0$ and
$\beta
\ge0$, the characteristic exponent of the ladder process $(\tau,H)$
is given by
%
%
\begin{equation}\label{fris}
\kappa(\alpha,\beta)=\exp\biggl(\int_0^\infty \,dt\int_{[0,\infty
)}(e^{-t}-e^{-\alpha t-\beta x})t^{-1} \p(X_t\in dx)\biggr).
\end{equation}
Note that the constant $k$, which appears in this theorem, is equal to
1, according to our normalization;
see Section~\ref{sec1}. Then recall the definition of $\kappa(\alpha,\beta)$:
$\e(e^{-\alpha\tau_u-\beta H_u})=e^{-u\kappa(\alpha,\beta
)}$. This expression
is differentiable, in $\alpha>0$ and in $u>0$. Differentiating first
both sides in $\alpha$, we obtain
%
%
\begin{equation}\label{6621}
\e(\tau_ue^{-\alpha\tau_u-\beta H_u})=u \e(e^{-\alpha\tau
_u-\beta
H_u})\frac{\partial}{\partial\alpha}\kappa(\alpha,\beta).
\end{equation}
Then since
\begin{eqnarray*}
&&\frac{\partial}{\partial\alpha}\int_0^\infty \,dt\int_{[0,\infty
)}(e^{-t}-e^{-\alpha t-\beta x})t^{-1} \p(X_t\in dx)\\
&&\qquad=\int_0^\infty
\,dt\int_{[0,\infty)}e^{-\alpha t-\beta x} \p(X_t\in dx)\\
&&\qquad=\int_0^\infty e^{-\alpha t}\e\bigl(e^{-\beta X_t}\ind_{\{X_t\ge0\}
}\bigr)\,dt,
\end{eqnarray*}
we derive from (\ref{6621}) and (\ref{fris}) that
\begin{eqnarray*}
\e(\tau_ue^{-\alpha\tau_u-\beta H_u})&=&-u\int_0^\infty e^{-\alpha
t}\e
\bigl(e^{-\beta X_t}\ind_{\{X_t\ge0\}}\bigr)\,dt\frac{\partial}{\partial
u}\e(e^{-\alpha\tau_u-\beta H_u})\\
&=&-u\frac{\partial}{\partial u}\int_0^\infty\e\bigl(\e(e^{-\alpha
\tau
_u-\beta H_u})e^{-\alpha t}e^{-\beta X_t}\ind_{\{X_t\ge0\}}\bigr)\,dt.
\end{eqnarray*}
Let $\tilde{X}$ be a copy of $X$ which is independent of $(\tau
_u,H_u)$. Then the above expression may be written as
\[
\e(\tau_ue^{-\alpha\tau_u-\beta H_u})
=-u\frac{\partial}{\partial u}\e\biggl(\int_0^\infty\exp\bigl(-\alpha
(t+\tau
_u)-\beta(\tilde{X}_t+H_u)\bigr)\ind_{\{\tilde{X}_t\ge0\}} \,dt\biggr).
\]
For $\tilde{X}$, we may take, for instance,
$\tilde{X}=(X_{\tau_u+t}-X_{\tau_u}, t\ge0)$, so that it follows from
a change of variables and
the definition of $(\ell_x, x\ge0)$,
\begin{eqnarray*}
\e(\tau_ue^{-\alpha\tau_u-\beta H_u})&=&-u\frac{\partial
}{\partial u}\e
\biggl(\int_0^\infty\exp\bigl(-\alpha(t+\tau_u)-\beta X_{\tau_u+t}\bigr)\ind_{\{
X_{\tau_u+t}\ge H_u\}} \,dt\biggr)\\
&=&-u\frac{\partial}{\partial u}\e\biggl(\int_0^\infty\exp(-\alpha
t-\beta X_{t})\ind_{\{X_{t}\ge H_u, \tau_u\le t\}} \,dt\biggr)\\
&=&-u\frac{\partial}{\partial u}\int_0^\infty \,dt e^{-\alpha t}\int
_{[0,\infty)} e^{-\beta x}\p(X_t\in dx,\ell_x> u),
\end{eqnarray*}
from which we deduce that
\begin{eqnarray*}
&&\int_{[0,\infty)^2}e^{-\alpha t-\beta x} t \p(\tau_u\in dt,H_u\in
dx) \,du\\
&&\qquad=
\int_{[0,\infty)^2} e^{-\alpha t-\beta x}u \p(X_t\in dx,\ell_x\in du)
\,dt,
\end{eqnarray*}
and (\ref{4594}) follows by inverting the Laplace transforms.

Let $\mathbf{e}$ be an exponentially distributed random variable with
parameter $q$, which is independent of $X$.
From identity (6.18), page 159 in~\cite{ky}, we have
%
%
\begin{equation}\label{fris1}
\qquad\quad\e(\exp(-\beta\overline{X}_\mathbf{e}))=\kappa(q,0)\int
_{[0,\infty)^2}e^{-q t-\beta x}
\int_0^\infty\p(\tau_s\in dt,H_s\in dx) \,ds.
\end{equation}
Suppose that $X$ is of type 1 or 2. By taking the Laplace transforms in
$x$ and $t$ of identity (\ref{both1}) in Corollary
\ref{law1}, we obtain
%
%
\begin{equation}\label{3137}
\e(\exp(-\beta\overline{X}_\mathbf{e}))=\bigl(q\mathtt{d}+n(1-e^{-q\zeta})\bigr)
n^*\biggl(\int_0^\zeta e^{-q s}e^{-\beta\varepsilon_s} \,ds\biggr),
\end{equation}
and by comparing (\ref{3247}), (\ref{fris1}) and (\ref{3137}), it follows
%
%
\begin{equation}\label{1247}
\quad\qquad n^*\biggl(\int_0^\zeta e^{-q s}e^{-\beta\varepsilon_s} \,ds\biggr)=\int
_{[0,\infty)^2}e^{-q t-\beta x}
\int_0^\infty\p(\tau_s\in dt,H_s\in dx) \,ds.
\end{equation}
Then we derive part 2 from (\ref{1247}), (\ref{4594}) and uniqueness of
the Laplace transform. If $X$ is of type 3, then taking
the Laplace transforms in $x$ and $t$ of identity (\ref{both1}), gives
%
%
\begin{equation}\label{3138}
\e(\exp(-\beta\overline{X}_\mathbf{e}))= n(1-e^{-q\zeta})
\biggl(\mathtt{d}^*+n^*\biggl(\int_0^\zeta e^{-q t}e^{-\beta\varepsilon_t}
\,dt\biggr)\biggr),
\end{equation}
so that by comparing (\ref{3247}), (\ref{fris1}) and (\ref{3138}),
we obtain
%
%
\begin{eqnarray}\label{1250}
&&\mathtt{d}^*+n^*\biggl(\int_0^\zeta e^{-q t}e^{-\beta\varepsilon_t} \,dt
\biggr)
\nonumber
\\[-8pt]
\\[-8pt]
\nonumber
&&\qquad=\int_{[0,\infty)^2}e^{-q t-\beta x}
\int_0^\infty\p(\tau_s\in dt,H_s\in dx) \,ds,
\end{eqnarray}
and part 2 follows from (\ref{1250}) and (\ref{4594}) in this case.

Then we show the third assertion. First note that $q_t^*$ is absolutely
continuous with respect to $p_t^+$ for all $t>0$,
since from (\ref{2681}) we have for any Borel set $B\subset\mathbb
{R}_+$ such that $\p(X_t\in B)=0$,
\[
q_t^*(B)=t^{-1}\e\bigl(\ell(X_t)\ind_{\{X_t\in B\}}\bigr)=0.
\]
Conversely, take a Borel set $B\subset\mathbb{R}_+$, such that $\p
(X_t\in B)>0$. Then since $\p(X_t=0)=0$, there exists $y>0$ such that
$\p(X_t\in B, X_t>y)>0$. As the right continuous inverse of a subordinator,
$(\ell_x)$ is nondecreasing and we have for all $x>0$, $\p(\ell_x>0)=1$.
Therefore the result follows from the inequality
\[
0<\e\bigl(\ell_{y}\ind_{\{X_t\in B, X_t>y\}}\bigr)\le\e\bigl(\ell(X_t)\ind_{\{
X_t\in
B\}}\bigr),
\]
together with identity (\ref{2681}).
\end{pf}

Recall from Section~\ref{prelim} that $\pi$ is the L\'evy
measure of the ladder time process $\tau$ and that
$\overline{\pi}(t)=\pi(t,\infty)$.

\begin{lemma}\label{equivalence2} Under the assumption of Lemma $\ref
{equivalence1}$, for all $t>0$,
the following measures of $\mathbb{R}_+$:
\[
\int_0^t\overline{\pi}(t-s)q_s^*(dx) \,ds \quad\mbox{and} \quad\int
_0^tq_s^*(dx) \,ds
\]
are equivalent.
\end{lemma}
\begin{pf} For all Borel set $B\subset\mathbb{R}_+$, we have
\[
\overline{\pi}(t)\int_0^tq_s^*(B) \,ds\le
\int_0^t\overline{\pi}(t-s)q_s^*(B) \,ds,
\]
hence $\int_0^tq_s^*(dx) \,ds$ is absolutely continuous with respect to
$\int_0^t\overline{\pi}(t-s)q_s^*(dx) \,ds$. Moreover,
for all $q\in(0,t)$ and all Borel set $B\subset\mathbb{R}_+$, we may write
\begin{eqnarray*}
\int_0^t\overline{\pi}(t-s)q_s^*(B) \,ds\le
\overline{\pi}(q)\int_0^tq_s^*(B) \,ds+
\int_{t-q}^t\overline{\pi}(t-s)q_s^*(B) \,ds<\infty.
\end{eqnarray*}
Hence if $\int_0^tq_s^*(B) \,ds=0$, then for all $q\in(0,t)$,
\[
\int_0^t\overline{\pi}(t-s)q_s^*(B) \,ds\le
\int_{t-q}^t\overline{\pi}(t-s)q_s^*(B) \,ds<\infty.
\]
The finiteness of the right-hand side of the above inequality can be
derived from
relation (\ref{pi}) and Corollary~\ref{law1}. Hence this term tends to
0 as $q$ tends to 0,
so that the equivalence between the measures $\int_0^t\overline{\pi
}(t-s)q_s^*(dx) \,ds$
and $\int_0^tq_s^*(dx) \,ds$ is proved.
\end{pf}

Now we are ready to prove all the results of Section \ref
{main}.

\begin{pf*}{Proof of Theorem \protect\ref{type}} When $X$ is of type 1 or
2, it follows from Corollary~\ref{law1},
part 3 of Lemma~\ref{equivalence1}, relation (\ref{pi}) and Lemma
\ref
{equivalence2}. When $X$ is of type~3, the arguments
are the same, except that one has to take account of the fact that the
law of $\overline{X}_t$ has an atom at 0, as
it is specified in Corollary~\ref{law1}.~%
\end{pf*}

\begin{pf*}{Proof of Theorem \protect\ref{th2}} We first prove that part
2 implies part 1. To that aim, observe that
\[
\overline{X}_{2t}=\max\Bigl\{\overline{X}_t,X_t+\sup_{0\le s\le
t}(X_{t+s}-X_t)\Bigr\}=\max\Bigl\{\overline{X}_t,X_t+\sup_{0\le s\le
t}X^{(1)}_s\Bigr\},
\]
where $X^{(1)}$ is an independent copy of $X$. From this independence
and the above expression, we easily deduce that if the law of
$\overline{X}_t$ is absolutely continuous, then so is this of
$\overline
{X}_{2t}$. Therefore, from Theorem~\ref{type}, the measure
$\mu^+_{2t}$ is absolutely continuous. This clearly implies that for
all $s\in(0,2t]$, the measure $\mu_s^+$ is absolutely continuous.
Applying Theorem~\ref{type} again, it follows that the law of
$\overline
{X}_s$ is absolutely continuous, for all $s\in(0,2t]$. Then
we show the desired result by reiterating this argument. So, part 1 is
equivalent to part 2.

Let us assume that part 1 holds. Then for all $t>0$, the law of
$\overline{X}_t$ is absolutely continuous. Therefore the resolvent measure
$U(dx)$ is absolutely continuous. Indeed,
let $\mathbf{e}$ be an independent exponentially distributed random time
with parameter 1,
then the law of $\overline{X}_\mathbf{e}$ admits a density, hence the law
of $X_\mathbf{e}=X_\mathbf{e}-\overline{X}_\mathbf{e}+\overline
{X}_\mathbf{e}$
also admits a density, since the random variables $X_\mathbf
{e}-\overline
{X}_\mathbf{e}$ and $\overline{X}_\mathbf{e}$ are independent; see
Chapter VI
in~\cite{be}. Since the law of $X_\mathbf{e}$ is precisely the measure
$U(dx)$, we have proved that part 1 implies part~4.
Then part 4 implies part 3 and from Corollary~\ref{lebesgue}, part 3
implies part 1.

It remains to show that part 5 is equivalent to part 1. To this aim,
first observe that $V(dx)$ is absolutely continuous if and only
if $\int_0^tq_s^*(dx) \,ds$ is absolutely continuous, for all $t>0$. Indeed,
from part 2 of Lemma~\ref{equivalence1}, we have $V(dx)=\int_0^\infty
q_s^*(dx) \,ds$, and hence if $V(dx)$ is absolutely continuous, then so
are the measures $\int_0^t q_s^*(dx) \,ds$, for all $t>0$. Conversely
assume that the measures $\int_0^t q_s^*(dx) \,ds$ are absolutely continuous
for all $t>0$. Let $A$ be a Borel set of $\mathbb{R}_+$ such that
$\lambda_+(A)=0$. From the assumption, $q_s^*(A)=0$, for $\lambda
$-almost every
$s>0$, hence $V(A)=\int_0^\infty q_s^*(A) \,ds=0$, so that $V(dx)$ is
absolutely continuous. Then from Lemma~\ref{equivalence2} and Corollary
\ref{law1}, for each $t$, the law of $\overline{X}_t$ is equivalent to
the measure $\int_0^t q_s^*(dx) \,ds$. Therefore part 5 is equivalent
to part 1.\vadjust{\goodbreak}
\end{pf*}

\begin{pf*}{Proof of Theorem \protect\ref{coro2}} If $p_t^+\ll\lambda^+$
for all $t>0$, then from part 3 of Lemma~\ref{equivalence1}, $q_t^*\ll
\lambda^+$, for all $t>0$. Suppose moreover that 0 is regular for
$(0,\infty)$, and let $A$ be a Borel subset of $\mathbb{R}$,
such that $\lambda(A)=0$. Then from Corollary~\ref{semigroup} and
Fubini theorem, we have
\[
p_t(A)=\int_0^t \,ds q_s^**\bar{q}_{t-s}(A)+\mathtt{d} q_t^*(A),
\]
where $\bar{q}_s$ and $q_s^*$ are extended on $\mathbb{R}$, as in this
corollary.
But from the assumptions, for all $0<s<t$, $q_s^**\bar{q}_{t-s}(A)=0$
and $q_t^*(A)=0$, hence $p_t(A)=0$,
for all $t>0$ and $p_t$ is absolutely continuous, for all $t>0$. So
part 1 implies part 2, and the converse is obvious.

Then it readily follows from part 3 of Lemma~\ref{equivalence1} and
identity (\ref{26811}) that part 1 implies part 3 (recall that $\mathtt{d}^*=0$ in
the present case). Now suppose that $V(dt,dx)$ is absolutely continuous
with respect to the Lebesgue measure on $\mathbb{R}_+^2$. Then we
derive from identity (\ref{26811}) that the measures $q_t^*(dx)$ are
absolutely continuous for $\lambda$-almost every $t>0$. From Corollary
\ref{semigroup}, it means that $p_t$ is absolutely continuous for
$\lambda$-almost every $t>0$. But if the semigroup $p_t$ is absolutely
continuous for some~$t$, then $p_s$ is absolutely continuous for all
$s\ge t$. Hence $p_t$ is actually absolutely continuous, for all $t>0$,
and part 3 implies part 2.

Then suppose that $X$ is of type 1, and recall that $\mathtt{d}=\mathtt{d}^*=0$ in this case. From Theorem~\ref{law} and part 2 of
Lemma~\ref{equivalence1}, we have
%
%
\begin{equation}\label{potential}
\p(g_t\in ds,\overline{X}_t\in dx)= n(t-s<\zeta)V(ds,dx).
\end{equation}
Since $n(t-s<\zeta)>0$, for all $s\in[0,t]$, we easily derive from
identity (\ref{potential}) that part 3 and part 4 are equivalent.

Let us denote by $p_t^-$ the restriction of $p_t$ to $\mathbb{R}_-$. If
part 2 is satisfied, then $p_t^+$ and $p_t^-$ are absolutely continuous
for all $t>0$. Then from part 3 of Lemma~\ref{equivalence1} applied to
$X$ and its dual process $-X$, it follows that $q_t$ and $q_t^*$ are
absolutely continuous for all $t>0$, so that from Theorem~\ref{law},
the triple $(g_t,\overline{X}_t,X_t)$ is absolutely continuous for all
$t>0$; hence part 2 implies part~5.
Then part 5 clearly implies part 4.
\end{pf*}

\begin{pf*}{Proof of Proposition \protect\ref{coro3}} In this proof, it
suffices to assume that $Y$ is
a deterministic process, that is, $(T_n)$, $(a_n)$ and $(b_n)$ in (\ref
{elem}) are deterministic sequences.

In order to prove part 1, let us first assume that $a_n=0$, for all
$n$. Then recall that from Theorem~\ref{th2}, the law
of $\overline{X}_t$ is absolutely continuous, for all $t>0$.
Fix $t>0$ and let $n$ be such that $t\in[T_n,T_{n+1})$. Set
$Z_k=Y_{T_k}+\sup_{T_k\le s<T_{k+1}}X_s$ and
$Z=Y_{T_n}+\sup_{T_n\le s<t}X_s$, and then we have
%
%
\begin{equation}\label{4580}\sup_{s\le t}X_s+Y_s=\max\{Z_1,Z_2,\ldots,Z_{n-1},Z\}.
\end{equation}
But we can write
%
%
\begin{equation}\label{4380}
 Z_k=Y_{T_k}+X_{T_k}+\sup_{s\le T_{k+1}-T_k}X^{(k)}_s\quad\!\! \mbox{and}\quad\!\!
Z=Y_{T_n}+
X_{T_n}+\sup_{s\le t-T_n}X^{(n)}_s,\hspace*{-35pt}
\end{equation}
where $X^{(k)}$, $k=1,\ldots,n$ are copies of $X$ such that $X,
X^{(k)}$, $k=1,\ldots,n$ are independent.
From Theorem~\ref{th2}, the laws of $\sup_{s\le T_{k+1}-T_k}X^{(k)}_s$
and\break $\sup_{s\le t-T_n}X^{(n)}_s$ are absolutely continuous.
From the representation (\ref{4380}) and the independence hypothesis,
we derive that the laws of $Z_1,Z_2,\ldots,Z_{k-1}$ and $Z$ are
absolutely continuous. Since the maximum of any finite sequence of
absolutely continuous random variables is itself
absolutely continuous, we conclude that the law of $\sup_{s\le
t}X_s+Y_s$ is absolutely continuous, and the first part is proved.

Now we assume that $(a_n)$ is any deterministic sequence. Then we have
(\ref{4580}) with
%
%
\begin{eqnarray}\label{4381}
Z_k&=&b_k+X_{T_k}+\sup_{s\le T_{k+1}-T_k}X^{(k)}_s+a_ks\quad \mbox{and}
\nonumber
\\[-8pt]
\\[-8pt]
\nonumber
Z&=&b_n+
X_{T_n}+\sup_{s\le t-T_n}X^{(n)}_s+a_ns,
\end{eqnarray}
where $X^{(k)}$, $k=1,\ldots,n$ are as above. If $p^+_t\ll\lambda^+$
for all $t$, then this property also holds
for the process $X$ with any drift $a$, that is, $X_t+at$, so from
Theorem~\ref{type} the laws of $\sup_{s\le T_{k+1}-T_k}X^{(k)}_s+a_ks$
and $\sup_{s\le t-T_n}X^{(n)}_s+a_ns$ are absolutely continuous, and we
conclude that the law of $\sup_{s\le t}X_s+Y_s$ is
absolutely continuous, in the same way as for the first part.

Finally, if $X$ has unbounded variations, then it is of type 1. If
moreover, for instance, the ladder
height process at the supremum $H$ has a positive drift, then from
Theorem~\ref{th2} and the remark thereafter, the law of $\overline
{X}_t$ is absolutely continuous for all $t>0$. Since $X$ has unbounded
variations, it follows from (iv) page 64 in~\cite{do} that
for any $a\in\mathbb{R}$, the ladder height process at the supremum of
the drifted L\'evy process $X_t+at$ also has a positive drift,
and since $X_t+at$ is also of type 1, the law of $\sup_{s\le t}X_s+as$
is absolutely continuous. Then from Theorem~\ref{th2}, the laws of
$\sup
_{s\le T_{k+1}-T_k}X^{(k)}_s+a_ks$ and $\sup_{s\le
t-T_n}X^{(n)}_s+a_ns$ are absolutely continuous, and again we conclude
that the law of $\sup_{s\le t}X_s+Y_s$ is absolutely continuous, in the
same way as for the first part.
\end{pf*}

\begin{pf*}{Proof of Proposition \protect\ref{plsn}} Recall that under
the assumption of this proposition, we have $\mathtt{d}^*=0$.
So, we derive from Theorem~\ref{law}, by integrating identity (\ref
{both}) over $y$ and from part 2 of Lemma~\ref{equivalence1},
that
\begin{eqnarray*}
\p(g_t\in ds,\overline{X}_t\in dx)&= &s^{-1}n(t-s<\zeta)\e\bigl(\ell(x)\ind_{\{X_s\in dx\}}\bigr)\ind
_{(0,t]}(s) \,ds\\
&&{}+
\mathtt{d}\delta_{\{t\}}(ds)t^{-1}\e\bigl(\ell(x)\ind_{\{X_s\in dx\}}\bigr).
\end{eqnarray*}
Since $X$ has no positive jumps, then $\overline{X}_t$ continuous.
Moreover, it is an increasing additive functional
of the reflected process $\overline{X}_t-X_t$, such that
\[
\e\biggl(\int_0^\infty e^{-t}\, d\overline{X}_t\biggr)=\Phi(1)^{-1},
\]
where $\Phi$ is the Laplace exponent of the subordinator $T_x=\inf\{
t:X_t>x\}$. Hence we have
$L_t=c\overline{X}_t$, with $c=\Phi(1)$. Then it follows from the
definition of $H$ and $\ell$, that
\[
H_u=c^{-1}u,\qquad \mbox{on $H_u<\infty$}\quad \mbox{and}\quad \ell_x=cx
\qquad\mbox{on $\ell_x<\infty$.}
\]
Besides, from part 1 of Lemma~\ref{equivalence1}, we have by
integrating (\ref{4594}) over $u\in[0,\infty)$,
%
%
\begin{equation}\label{end}
cxp_t^+(dx) \,dt=ct\p(\tau_{cx}\in dt) \,dx,
\end{equation}
as measures on $[0,\infty)^2$. This ends the proof of the
proposition.
\end{pf*}

Note that identity (\ref{end}) may also be derived from
Corollary VII.3, page 190 in~\cite{be} or from Theorem 3 in~\cite{ac}.
The constant $c$ appearing in our expression is due to the choice of
the normalization of the local time in (\ref{norm1}).

\begin{pf*}{Proof of Corollary \protect\ref{counterexample}} We may check
that $\int_{(0,1)}x\Pi(dx)=\infty$ in both cases 1 and 2, so that $X$
has unbounded variation and
it is of type 1, from Rogozin's criterion; see~\cite{be}, page 167.

On the one hand, in part 1, since $X$ has no positive jumps, the ladder
height process $H$ is a pure drift, so it follows from Theorem
\ref{th2} and the remark thereafter that the law of $\overline{X}_t$ is
absolutely continuous for all $t>0$.
On the other hand, following~\cite{or}, we see that $-\log|\psi
(\lambda
)|$ does not tend to $+\infty$ as $|\lambda|\rightarrow\infty$,
so that from the Riemann--Lebesgue theorem, $p_t(dx)$ is not absolutely
continuous. But since $\Pi(dx)$ is discrete with infinite mass, it
follows from the Hartman--Wintner theorem (see Theorem 27.16 in \cite
{sa}) that $p_t(dx)$ is continuous singular.

Then in part 2, since $X$ is symmetric, it follows from the discussion
which comes just before Theorem~\ref{coro2} that the resolvent measure
$U(dx)$ of $X$ is not absolutely continuous, so the result follows from
Theorem~\ref{th2}.
\end{pf*}

In order to prove Theorem~\ref{new}, we need the following lemma. We
say that a sequence of random variables $S_1,\ldots,S_n,\ldots,$
with $S_0=0$ is a cyclically exchangeable chain if for any $n\ge1$, the
increments $S_1,S_2-S_1,\ldots,S_n-S_{n-1}$
are cyclically exchangeable. We emphasize that any random walk
satisfies this property.
For $x\ge0$ and $n\ge1$, we define $\overline{S}_n=\max_{k\le
n}S_k$ and
\[
\Lambda_n^x=\operatorname{Card} \{1\le k\le n:S_k=\overline{S}_k,
\overline
{S}_n-x\le\overline{S}_k\le\overline{S}_n\}.
\]
The random variable $\Lambda_n^{x}$ may be considered as a counting
measure of the times at which the chain reaches its past
maximum between the levels $\overline{S}_n-x$ and $\overline{S}_n$.
Note that $\Lambda_n^x\ge1$, a.s., whenever
$\p(S_1<0,S_2<0,\ldots,S_n<0)=0$.
\begin{lemma}\label{newlem}
Let $(S_n)_{n\ge0}$ be any cyclically exchangeable chain such that $\p
(S_1<0,S_2<0,\ldots,S_n<0)=0$,
then for any set $A\in\mathcal{ B}_{\mathbb{R}_+}$ and $n\ge1$,
\[
\e\bigl((\Lambda_n^{S_n})^{-1}\ind_{\{S_n=\overline{S}_n\in
A\}}\bigr)=\frac1n\p(S_n\in A).
\]
\end{lemma}
\begin{pf} Fix $A\in\mathcal{ B}_{\mathbb{R}_+}$, $n\ge1$, and let
$1\le k\le n$ be such that $\p(\Lambda_n^{S_n}=k)>0$.
Then note that conditionally on $\Lambda_n^{S_n}=k$ and $S_n\in A$, the
chain $S_0,S_1,\ldots,S_n$ is
cyclically exchangeable. Moreover, amongst the $n$ cyclical
permutations in the set of trajectories $\{S_0,\ldots,S_n:\Lambda
_n^{S_n}=k,S_n\in A\}$, there are exactly $k$ trajectories which satisfy
the condition $S_n=\overline{S}_n$. This proves the identity
\[
\frac1k\p(S_n=\overline{S}_n\in A | \Lambda_n^{S_n}=k)=\frac1n\p
(S_n\in A | \Lambda_n^{S_n}=k),
\]
and the result is obtained by integrating over the law of $\Lambda_n^{S_n}$.
\end{pf}

Lemma~\ref{newlem} can be compared to Corollary 1 in \cite
{acd}. The only difference is that in~\cite{acd}, only
strict records of the chain are considered, whereas in our case, the
``local time'' $\Lambda_n^{S_n}$ counts all the records
(i.e., weak records) between the levels $\overline{S}_n-S_n$ and
$\overline{S}_n$.

\begin{pf*}{Proof of Theorem \protect\ref{new}} Fix $t>0$. Equivalence
between the measures $p_t^+$, $\mu_t^+$, $U^+$ and $\upsilon^+$
simply follows from the following decompositions:
\begin{eqnarray*}
p_t^+(dx)&=&\sum_{n=0}^\infty\p(N_t=n)\p(S_n\in dx),\\
\mu_t^+(dx)&=&\sum_{n=0}^\infty\biggl(\int_0^t\p(N_s=n) \,ds\biggr)\p
(S_n\in dx),\\
U^+(dx)&=& \sum_{n=0}^\infty\biggl(\int_0^\infty e^{-s}\p(N_s=n)
\,ds\biggr)\p(S_n\in dx)
\end{eqnarray*}
and the definition of $\upsilon^+$.

It remains to prove the equivalence between $\p(\overline{X}_t\in dx)$
and $\upsilon^+(dx)$.
To this aim, note that $\overline{X}_t=\overline{S}_{N_t}$, so that
%
%
\begin{equation}\label{3034}
\p(\overline{X}_t\in dx)=\sum_{n=0}^\infty\p(N_t=n)\p(\overline
{S}_n\in dx).
\end{equation}
Let $A\in\mathcal{ B}_{\mathbb{R}_+}$ be such that $\upsilon^+(A)=0$, then
by definition of $\upsilon^+$,
$\p(S_n\in A)=0$, for all $n\ge0$. This implies that
%
%
\begin{equation}\label{3039}
\p(\overline{S}_n\in A)\le\sum_{k=0}^n\p(S_k\in A,S_k=\overline
{S}_k)=0,
\end{equation}
so that from (\ref{3034}), $\p(\overline{X}_t\in A)=0$.
Conversely, assume that $\p(\overline{X}_t\in A)=0$. Then from (\ref
{3034}), for any $n\ge0$,
$\p(\overline{S}_n\in A)=0$. For $n=0$, we have $S_0=\overline{S}_0=0$,
and for all $n\ge1$, from Lemma~\ref{newlem}, we have
\[
\e\bigl((\Lambda_n^{S_n})^{-1}\ind_{\{S_n=\overline{S}_n\in
A\}}\bigr)=\frac1n\p(S_n\in A)=0.
\]
We conclude that $\upsilon^+(A)=0$.
\end{pf*}

\section*{Acknowledgments}
I would like to thank Laurent Denis
who has brought the problem of the absolute continuity of the supremum
of L\'evy processes to my attention and for fruitful discussions on
this subject. I am also very grateful to Victor Rivero for some
valuable comments.



\printaddresses

\end{document}